\documentclass[12pt,a4paper]{amsart}
\usepackage{amsmath}
\usepackage{amsfonts}
\usepackage{amssymb}
\usepackage{amsthm}
\usepackage[utf8]{inputenc}
\usepackage{array}
\usepackage{fancyhdr}
\usepackage{geometry}
\usepackage{multicol}
\usepackage{enumerate}
\usepackage{stackrel}
\usepackage{color}

\newcommand{\B}{\mathbf{B}}		

\newcommand{\dv}{\operatorname{div}}

\renewcommand{\le}{\leqslant}
\renewcommand{\ge}{\geqslant}

\newcommand{\cN}{{\ensuremath{\mathcal N}}}

\newcommand{\R}{{\ensuremath{\mathbb R}}}

\newcommand{\dist}{\operatorname{dist}}

\newcommand{\loc}{\mathrm{loc}}

\def\barroman#1{\sbox0{#1}\dimen0=\dimexpr\wd0+1pt\relax
  \makebox[\dimen0]{\rlap{\vrule width\dimen0 height 0.06ex depth 0.06ex}%
    \rlap{\vrule width\dimen0 height\dimexpr\ht0+0.03ex\relax 
            depth\dimexpr-\ht0+0.09ex\relax}%
    \kern.5pt#1\kern.5pt}}

\def\XXint#1#2#3{{\setbox0=\hbox{$#1{#2#3}{\int}$}
     \vcenter{\hbox{$#2#3$}}\kern-.5\wd0}}

\newtheorem{tw}{Theorem}[section]
\newtheorem{lem}[tw]{Lemma}
\newtheorem{wn}[tw]{Corollary}

\newtheorem{uwaga}[tw]{Remark}

\theoremstyle{definition}
\newtheorem{df}{Definition}[section]

\binoppenalty=\maxdimen
\relpenalty=\maxdimen

\author[M. Miśkiewicz]{Michał Miśkiewicz}
\title[Fractional differentiability for the $p$-Laplace system]{Fractional differentiability for solutions of the~inhomogenous $p$-Laplace system}
\address{Institute of Mathematics, University of Warsaw,\newline Banacha 2, 02-097 Warszawa, Poland}
\email{m.miskiewicz@mimuw.edu.pl}
\thanks{The research has been supported by the NCN grant no. 2012/05/E/ST1/03232 (years 2013-2017).}


\begin{document}

\begin{abstract}
It is shown that if $p \ge 3$ and $u \in W^{1,p}(\Omega,\mathbb{R}^N)$ solves the inhomogenous $p$-Laplace system 
\[
\operatorname{div} (|\nabla u|^{p-2} \nabla u) = f, \qquad f \in W^{1,p'}(\Omega,\mathbb{R}^N),
\]
then locally the gradient $\nabla u$ lies in the fractional Nikol'ski{\u\i} space $\mathcal{N}^{\theta,2/\theta}$ with any $\theta \in [ \tfrac{2}{p}, \tfrac{2}{p-1} )$. To the author's knowledge, this result is new even in the case of $p$-harmonic functions, slightly improving known $\mathcal{N}^{2/p,p}$ estimates. The method used here is an extension of the one used by A. Cellina in the case $2 \le p < 3$ to show $W^{1,2}$ regularity. 
\end{abstract}

\maketitle

\section{Introduction}
\label{sec:intro}

Recall that for $p > 1$, a $p$-harmonic function is a minimizer of the Dirichlet $p$-energy functional $\tfrac 1p \int_\Omega |\nabla u|^p$ in the class $W^{1,p}(\Omega)$ with fixed Dirichlet boundary conditions. It is also a solution of the Euler-Lagrange equation $\dv ( |\nabla u|^{p-2} \nabla u ) = 0$. 
To the author's knowledge, some of the best known local regularity results for the gradient of a $p$-harmonic function $u \in W^{1,p}$ are: 
\begin{itemize}
\item
$\nabla u \in C^{0,\alpha}$ for $1 < p < \infty$ (Ural'tseva \cite{Ura} for $p \ge 2$, see also \cite{LadUra,Eva82,Lew,DiB,Uhl,Tol84}),

\item
$\nabla u \in W^{1,p}$ for $1 < p \le 2$ (see \cite{Lin}), 

\item
$\nabla u \in W^{1,2}$ for $2 \le p < 3$ (Cellina \cite{Cel}, Sciunzi \cite{Sci}), 

\item
$\nabla u \in \cN^{2/p,p}$ for $p \ge 2$ (Mingione \cite{Min}).
\end{itemize}
It is worth noting that most of them were obtained for more general second order operators, non-trivial source terms or in case of systems of equations. The Nikol'ski{\u\i} space $\cN^{\theta,q}$ mentioned in the last result is a variant of fractional Sobolev spaces (see Definition \ref{df:Nikolskij}) and it appears naturally in this context. The main result of this paper holds for solutions of the inhomogenous $p$-Laplace system, but to the author's knowledge it is new also in the case of $p$-harmonic functions. 

\begin{tw}
\label{th:p-harmonic-regularity}
Let $p \ge 3$ and assume that $u \in W^{1,p}(\Omega,\R^N)$ solves the system
\begin{equation}
\label{eq:p-Poisson}
\dv (|\nabla u|^{p-2} \nabla u^\alpha) = f^\alpha 
\qquad \text{in } \Omega \text{ for } \alpha = 1,\ldots,N, 
\end{equation}
where $f \in W^{1,p'}(\Omega,\R^N)$. Then $\nabla u \in \cN_\loc^{\theta,2/\theta}(\Omega,\R^N)$ for every $\theta \in [ \frac{2}{p}, \frac{2}{p-1} )$ with 
\[
\| \nabla u \|_{\cN^{\theta,2/\theta}(\Omega')} 
\le C \left(\| u \|_{W^{1,p}(\Omega)} + \| f \|^{\frac{1}{p-1}}_{W^{1,p'}(\Omega)} \right)
\qquad \text{for } \Omega' \Subset \Omega.
\]
Here and in the sequel, the constant $C$ may depend on the domains $\Omega',\Omega$, the dimensions $n,N$ and the parameters $p,\theta$, but not the functions involved. 
\end{tw}

\begin{uwaga}
A well known example (discussed in Section \ref{sec:sharpness}) shows the endpoint estimate $\cN^{\frac{2}{p-1},p-1}$ to be sharp: there is a solution of \eqref{eq:p-Poisson} satisfying $\nabla u \in \cN^{\frac{2}{p-1},p-1}$, but $\nabla u \notin \cN^{\frac{2}{p-1},q}$ for $q > p-1$ and $\nabla u \notin \cN^{\theta,p-1}$ for $\theta > \frac{2}{p-1}$. 
\end{uwaga}

\begin{uwaga}
Regularity of the source term $f$ is only needed for the estimate \eqref{eq:f-estimate}. A closer look reveals that for fixed $\theta$ it is enough to assume $f \in L^{p'}$ and $\nabla f \in L^r$ with $r = \frac{p}{2p - \frac{2}{\theta} - 1}$. Note that $r \searrow 1$ when $\theta \nearrow \frac{2}{p-1}$, so the assumptions are actually weaker for $\theta$ close to optimal. 
\end{uwaga}

Fractional differentiability estimates come from the following elementary observation: if $\beta$ is $\theta$-H{\"o}lder continuous and $V \in W^{1,2}$, then the composition $\beta(V)$ lies in $\cN^{\theta,2/\theta}$ (see Lemma \ref{lem:fract-diff} for the precise statement). 

In this context, recall a well-known result due to Bojarski and Iwaniec \cite{BojIwa}: if $u \in W^{1,p}$ is $p$-harmonic, then 
\[ V := |\nabla u|^{\frac{p-2}{2}} \nabla u \in W_\loc^{1,2}. \] 
One can recover $\nabla u$ from $V$ as $\nabla u = \beta(V)$, where $\beta(w) = |w|^{\frac{2}{p}-1} w$ is $\frac{2}{p}$-H{\"o}lder continuous, thus obtaining $\nabla u \in \cN_\loc^{2/p,p}$ as a~corollary. This was shown for a~quite general class of systems by Mingione \cite{Min}. Note that both proofs \cite{BojIwa,Min} rely on testing the equation with the same test function. 

Our aim is therefore to obtain $W^{1,2}$ estimates for some nonlinear expressions of the gradient -- similar to $V$, only with smaller exponents. In this way we are able to improve $\cN^{2/p,p}$ regularity of the gradient to almost $\cN^{\frac{2}{p-1},p-1}$. 

\begin{tw}
\label{th:nonlinear-reg}
Let $p \ge 3$ and assume that $u \in W^{1,p}(\Omega,\R^N)$ solves the $p$-Laplace system \eqref{eq:p-Poisson} with $f \in W^{1,p'}(\Omega,\R^N)$. Then 
\[
|\nabla u|^{s-1} \nabla u \in W_\loc^{1,2}(\Omega,\R^N)
\]
for each $\frac{p-1}{2} < s \le \frac{p}{2}$. Moreover, 
\[
\| |\nabla u|^{s-1} \nabla u \|_{W^{1,2}(\Omega')} \le 
C \left(\| u \|^s_{W^{1,p}(\Omega)} + \| f \|^{\frac{s}{p-1}}_{W^{1,p'}(\Omega)} \right)
\qquad \text{for } \Omega' \Subset \Omega.
\]
\end{tw}

\medskip

The proof follows roughly by differentiating the $p$-Laplace system \eqref{eq:p-Poisson} and testing the obtained system with the function $\eta^2 |\nabla u|^{2s-p} \nabla u$ ($\eta$ being a cut-off function). Since this process involves the second order derivatives of $u$, it cannot be carried out directly. 
The problem lies in the fact that for $p>2$ the $p$-Laplace system \eqref{eq:p-Poisson} is degenerate at points where $\nabla u = 0$. This difficulty is bypassed by approximating $u$ with solutions of some uniformly elliptic systems. 
For fixed $\varepsilon > 0$ we consider the following approximation of the Dirichlet $p$-energy functional: 
\[
F_\varepsilon(u) = \frac 1p \int_{\Omega} ( \varepsilon^2 + |\nabla u|^2 )^{p/2} + \int_\Omega \langle u,f_\varepsilon \rangle_{\R^N}, 
\]
where $f_\varepsilon$ is a smooth approximation of $f$. By standard theory, $F_\varepsilon$ has a unique smooth minimizer $u_\varepsilon \in u+W_0^{1,p}(\Omega,\R^N)$. 
Since the elliptic constant vanishes as $\varepsilon \to 0$, regularity of $u_\varepsilon$ might be lost in the limit, so our goal is to obtain estimates similar to those in Theorem \ref{th:nonlinear-reg} uniformly in $\varepsilon$ (this is done in Lemma~\ref{lem:unif-bounds}). 

\medskip

The method outlined above is an extention of the one employed by Cellina \cite{Cel} in the case $2 \le p < 3$. Indeed, the proof of Theorem \ref{th:nonlinear-reg} carries over to this case, leading to the following result. 

\begin{tw}
\label{th:nonlinear-reg-Cel}
Let $2 \le p < 3$ and assume that $u \in W^{1,p}(\Omega,\R^N)$ solves the $p$-Laplace system \eqref{eq:p-Poisson} with $f \in W^{1,p'}(\Omega,\R^N)$. Then 
\[
|\nabla u|^{s-1} \nabla u \in W_\loc^{1,2}(\Omega,\R^N)
\]
for each $1 \le s \le \frac{p}{2}$. Moreover, 
\[
\| |\nabla u|^{s-1} \nabla u \|_{W^{1,2}(\Omega')} \le C 
\left(\| u \|^s_{W^{1,p}(\Omega)} + \| f \|^{\frac{s}{p-1}}_{W^{1,p'}(\Omega)} \right)
\qquad \text{for } \Omega' \Subset \Omega.
\]
\end{tw}

We can take $s$ equal to $1$ in the above theorem, thus recovering the following result due to Cellina \cite{Cel}. In this case one does not need to use the fractional differentiability lemma (Lemma \ref{lem:fract-diff}). 

\begin{wn}[{\cite[Th.~1]{Cel}}]
\label{th:p-harmonic-regularity-Cel}
Let $2 \le p < 3$ and assume that $u \in W^{1,p}(\Omega,\R^N)$ solves the $p$-Laplace system \eqref{eq:p-Poisson} with $f \in W^{1,p'}(\Omega,\R^N)$. Then $\nabla u \in W_\loc^{1,2}(\Omega,\R^N)$ with 
\[
\| \nabla u \|_{W^{1,2}(\Omega')} \le C 
\left(\| u \|_{W^{1,p}(\Omega)} + \| f \|^{\frac{1}{p-1}}_{W^{1,p'}(\Omega)} \right)
\qquad \text{for } \Omega' \Subset \Omega.
\]
\end{wn}

\medskip

For the sake of clarity, the following exposition is restricted to the case $N=1$, i.e. to the single $p$-Laplace equation. The general case follows exactly the same lines, but one has to keep track of the additional indices. 

\section{Fractional Sobolev spaces}

The main result is concerned with the estimates in Nikol'ski{\u\i} spaces \cite{Nik} (see also \cite{Ada}), which we now define. Below $\Omega \subseteq \R^n$ is an open domain and for each $\delta > 0$ we denote $\Omega_\delta = \{ x \in \Omega : \B(x,\delta) \subseteq \Omega \}$. 

\begin{df}
\label{df:Nikolskij}
Let $u \in L^q(\Omega)$, $\theta \in [0,1]$. The Nikol'ski{\u\i} seminorm $[u]_{\cN^{\theta,q}(\Omega)}$ is defined as the smallest constant $A$ such that 
\[
\left( \int_{\Omega_{|v|}} |u(x+v)-u(x)|^q \right)^{1/q} \le A |v|^\theta 
\]
holds for all vectors $v \in \R^n$ of length $|v| \le \delta$. The norm in $\cN^{\theta,q}(\Omega)$ is 
\[
\| u \|_{\cN^{\theta,q}(\Omega)} := \| u \|_{L^q(\Omega)} + [ u ]_{\cN^{\theta,q}(\Omega)}.
\]
Changing the value of $\delta > 0$ amounts to choosing an equivalent norm. 
\end{df}

In the context of this paper, only local results are available due to the use of cut-off functions. Therefore we may fix a~subdomain $\Omega' \Subset \Omega$, choose $\delta = \dist(\Omega',\partial \Omega)$ and look for estimates of the form 
\[
\left( \int_{\Omega'} |u(x+v)-u(x)|^q \right)^{1/q} \le A |v|^\theta 
\qquad \text{for vectors of length } |v| \le \delta.
\]

\medskip

Note that the seminorms $\cN^{1,q}$ and $W^{1,q}$ are equivalent for $q>1$ due to the difference quotient characterization of Sobolev spaces. 
This will be exploited in Lemma \ref{lem:fract-diff}. Other basic examples are $\cN^{0,q} = L^q$ and $\cN^{\theta,\infty} = C^{0,\theta}$. For the sake of comparison, let us also mention the embeddings 
\[
\cN^{\theta+\varepsilon,q}(\Omega) \hookrightarrow W^{\theta,q}(\Omega) \hookrightarrow \cN^{\theta,q}(\Omega)
\]
valid for any $\varepsilon > 0$ \cite[7.73]{Ada}. Here $W^{\theta,q}$ stands for the fractional Slobodecki{\u\i}-Sobolev space. 

\section{Regularity of nonlinear expressions}

Let us introduce a slight change of notation. The functions $u$, $f$ solving the degenerate equation \eqref{eq:p-Poisson} shall be henceforth referred to as $u_0$, $f_0$. For $\varepsilon > 0$ we introduce $u_\varepsilon$, $f_\varepsilon$ as smooth solutions to a non-degenerate approximate equation. 

\medskip

Since the claim is local, we can assume without loss of regularity that the domain $\Omega \subseteq \R^n$ is bounded. For fixed $\varepsilon > 0$ we consider the following approximations: 
\begin{alignat*}{3}
l_\varepsilon(w) & = ( \varepsilon^2 + |w|^2 )^{1/2} 
&& \qquad \text{for } w \in \R^n, \\
L_\varepsilon(w) & = \tfrac{1}{p} l_\varepsilon(w)^p 
&& \qquad \text{for } w \in \R^n, \\
F_\varepsilon(u) & = \int_{\Omega} L_\varepsilon(\nabla u) + u f_\varepsilon 
&& \qquad \text{for } u \in u_0 + W_0^{1,p}(\Omega). 
\end{alignat*}
We choose $f_\varepsilon$ to be some family of smooth functions such that $f_\varepsilon \to f_0$ in $W^{1,p'}(\Omega)$. 
Taking the limit $\varepsilon \to 0$, one recovers the $p$-energy $F_0$. 

We begin by noting some basic properties needed in the sequel. 

\begin{lem}
\label{lem:L-properties}
For $l_\varepsilon$, $L_\varepsilon$ defined as above, 
\begin{enumerate}[(a)]
\item $\max(\varepsilon, |w|) \le l_\varepsilon(w) \le \varepsilon+|w|$, 
\item $l_\varepsilon(w) \searrow |w|$ as $\varepsilon \searrow 0$, 
\item $L_\varepsilon$ is smooth and 
\[
\frac{\partial^2 L_\varepsilon}{\partial w_i \partial w_j} (w)
= l_\varepsilon(w)^{p-2} \delta_{ij} + (p-2)l_\varepsilon(w)^{p-4} w_i w_j, 
\]
hence it is uniformly elliptic: 
\[
( \varepsilon^2 + |w|^2 )^{\frac{p-2}{2}} |v|^2 
\le \frac{\partial^2 L_\varepsilon}{\partial w_i \partial w_j} (w) v_i v_j 
\le (p-1) ( \varepsilon^2 + |w|^2 )^{\frac{p-2}{2}} |v|^2 
\]
holds for any $v,w \in \R^n$. 
\end{enumerate}
\end{lem}

The straightforward computations behind Lemma \ref{lem:L-properties} are omitted; these and later computations can be simplified by noting that 
\begin{alignat*}{3}
\frac{\partial l_\varepsilon}{\partial w_i} (w) 
& = l(w)^{-1} w_i 
&& \quad \text{for } w \in \R^n, \\
\frac{\partial}{\partial x_i} \left( l_\varepsilon(\nabla u)^s \right) 
& = s \cdot l_\varepsilon(\nabla u)^{s-2} \langle \nabla u_{x_i}, \nabla u \rangle 
&& \quad \text{for } u \colon \Omega \to \R. 
\end{alignat*}
An useful remark here is that the outcome of all computations depends on $\varepsilon$ only via the function $l_\varepsilon$, allowing us to show estimates uniform in $\varepsilon$.  

\medskip

The regularity result in Theorem \ref{th:nonlinear-reg} shall be first shown for similar nonlinear expressions of the gradients of the approximate solutions. For fixed parameters $s,\varepsilon > 0$ let us introduce the smooth function 
\[
\alpha_\varepsilon^s \colon \R^n \to \R^n, \qquad
\alpha_\varepsilon^s(w) = l_\varepsilon(w)^{s-1} w 
\]
Notice that for $\varepsilon = 0$ we recover the familiar expression $\alpha_0^s(w) = |w|^{s-1}w$ together with its inverse $\alpha_0^{1/s}$. 

\begin{lem}
\label{lem:unif-bounds}
Fix a solution $u_0 \in W^{1,p}(\Omega)$ of the equation \eqref{eq:p-Poisson} with $f_0 \in W^{1,p'}(\Omega)$. For each $\varepsilon \in (0,1)$ the functional $F_\varepsilon$ has a~unique smooth minimizer $u_\varepsilon \in u_0 + W_0^{1,p}(\Omega)$. Moreover, 
\begin{enumerate}[(a)]
\item the functions $u_\varepsilon$ are uniformly bounded in $W^{1,p}(\Omega)$, 
\item for each $\frac{p-1}{2} < s \le \frac{p}{2}$, the functions $\alpha_\varepsilon^s(\nabla u_\varepsilon)$ are uniformly bounded in $W_\loc^{1,2}(\Omega)$ with respect to $\varepsilon$, i.e. 
\[
\| \alpha_\varepsilon^s(\nabla u_\varepsilon) \|_{W^{1,2}(\Omega')} \le C(\Omega,\Omega',n,p,s,\| u_0 \|_{W^{1,p}(\Omega)},\| f_0 \|_{W^{1,p'}(\Omega)}) 
\qquad \text{for } \Omega' \Subset \Omega.
\]
\end{enumerate}
\end{lem}

\begin{proof}[Proof of part (a)]
The existence of unique minimizer $u_\varepsilon \in W^{1,p}(\Omega)$ is a standard result, and $C^{1,\alpha}$ regularity was shown by Tolksdorf \cite{Tol83} (also in the case of systems of equations). Since the resulting elliptic equation is non-degenerate, $u_\varepsilon$ is smooth by a~bootstrap argument (although only $C^2$ regularity is needed in the sequel). 

We turn our attention to the uniform $W^{1,p}$ estimates. First, $u_\varepsilon - u_0 \in W_0^{1,p}(\Omega)$, hence 
\begin{align*}
\| u_\varepsilon \|_{L^p(\Omega)} 
& \le \| u_\varepsilon - u_0 \|_{L^p(\Omega)} + \| u_0 \|_{L^p(\Omega)} \\
& \le C \| \nabla u_\varepsilon - \nabla u_0 \|_{L^p(\Omega)} + \| u_0 \|_{L^p(\Omega)} \\
& \le C \| \nabla u_\varepsilon \|_{L^p(\Omega)} + C \| u_0 \|_{W^{1,p}(\Omega)}
\end{align*}
by Poincar{\'e}'s inequality; thus we only need to bound $||\nabla u_\varepsilon||_{L^p(\Omega)}$. Using the minimality of $u_\varepsilon$ and the monotonicity from Lemma \ref{lem:L-properties}, we obtain the bound 
\[
\int_\Omega \tfrac 1p |\nabla u_\varepsilon|^p + u_\varepsilon f_\varepsilon \le F_\varepsilon(u_\varepsilon) \le F_\varepsilon(u_0) \le \int_\Omega \tfrac 1p (1+|\nabla u_0|^2)^{p/2} + u_0 f_\varepsilon, 
\]
which together with the previous one yields a uniform bound for $\| \nabla u_\varepsilon \|_{L^p(\Omega)}$. 
\end{proof}

Part (b) of Lemma \ref{lem:unif-bounds} is the key part of this paper; it will be proved in Section~\ref{ch:apriori}. Taking it for granted, we can pass to the limit and prove Theorem \ref{th:nonlinear-reg}. 

\begin{proof}[Proof of Theorem \ref{th:nonlinear-reg}]
By Lemma \ref{lem:unif-bounds}a we can choose a sequence $\varepsilon \searrow 0$ such that $u_\varepsilon$ converges weakly in $W^{1,p}(\Omega)$ to some $\bar{u}$, in particular $\bar{u} = u_0$ on $\partial \Omega$. It also shows that the linear parts of the functionals $F_0$, $F_\varepsilon$ converge: 
\[
\int_\Omega u_\varepsilon f_0, \ \int_\Omega u_\varepsilon f_\varepsilon \to \int_\Omega \bar{u} f_0.
\]
As for the nonlinear part, we argue again by minimality and monotonicity: 
\begin{align*}
F_0(\bar{u}) 
& \le \liminf_{\varepsilon \to 0} F_0(u_\varepsilon) \\
& \le \liminf_{\varepsilon \to 0} F_\varepsilon(u_\varepsilon) \\
& \le \liminf_{\varepsilon \to 0} F_\varepsilon(u_0) \\
& = F_0(u_0). 
\end{align*}
Recall that the solution $u_0$ of the $p$-Laplace system \eqref{eq:p-Poisson} is unique and easily seen to minimize the $p$-energy $F_0$. Hence $\bar{u}$ has to coincide with $u_0$ as another minimizer of $F_0$. 

\medskip

After fixing $\Omega' \Subset \Omega$, we use Lemma \ref{lem:unif-bounds}b in a similar way, obtaining $\alpha_\varepsilon^s(\nabla u_\varepsilon) \to \bar{\alpha}$ weakly in $W^{1,2}(\Omega')$ and a.e. We can assume that $\bar{\alpha} = \alpha_0^s(v)$ for some vector field $v$, as $\alpha_0^s$ is invertible. An elementary pointwise reasoning shows that the convergence $\alpha_\varepsilon^s(\nabla u_\varepsilon) \to \alpha_0^s(v)$ leads to $\nabla u_\varepsilon \to v$ a.e. Combining this with weak convergence $\nabla u_\varepsilon \to \nabla u_0$, we infer that $v = \nabla u_0$ and $\bar{\alpha} = \alpha_0^s(\nabla u_0)$, in consequence 
\[
\| \alpha_0^s(\nabla u_0) \|_{W^{1,2}(\Omega')} 
\le C(\Omega,\Omega',n,p,s,\| u_0 \|_{W^{1,p}(\Omega)},\| f_0 \|_{W^{1,p'}(\Omega)}).
\]

To show that the constant has the desired form, we note the scaling properties of the $p$-Laplace system \eqref{eq:p-Poisson}. For each $\lambda>0$, the functions $\lambda u_0$ and $\lambda^{p-1} f_0$ also solve \eqref{eq:p-Poisson}; let us choose $\lambda$ small enough so that their norms do not exceed $1$. Then by the above discussion $\| \alpha_0^s(\nabla (\lambda u_0)) \|_{W^{1,2}(\Omega')} \le C$, where $C$ is independent of the functions involved. Since $\alpha_0^s$ is $s$-homogenous, this yields $\| \alpha_0^s(\nabla u_0) \|_{W^{1,2}(\Omega')} \le C \lambda^{-s}$, which is equivalent to our claim. 
\end{proof}

\section{A priori estimates}
\label{ch:apriori}

Throughout this section, the value of $\varepsilon > 0$ is fixed and the subscript $\varepsilon$ is omitted in $u_\varepsilon,f_\varepsilon,l_\varepsilon,L_\varepsilon,F_\varepsilon,\alpha_\varepsilon^s$. 

\begin{proof}[Proof of Lemma \ref{lem:unif-bounds}b]
Fix the subdomain $\Omega' \Subset \Omega$ and a cut-off function $\eta \in C^\infty_c(\Omega)$ such that $\eta \equiv 1$ on $\Omega'$ and $\eta \ge 0$. Choose the parameter $\frac{p-1}{2} < s \le \frac{p}{2}$ and additionally denote $q = p-2s+2$, thus $2 \le q < 3$. 

Since $u$ is a smooth minimizer of $F$, it satisfies the Euler-Lagrange equation $\dv (\nabla L(\nabla u)) = f$ and also the differentiated system 
\[ \dv (D^2 L(\nabla u) \nabla u_{x_j}) = f_{x_j} \qquad \text{for } j=1,2,\ldots,n. \]
This system can be tested with the vector-valued function $\gamma = l(\nabla u)^{2-q} \nabla u$ multiplied by the cut-off function $\eta^2$, resulting in 
\begin{align*}
\int_\Omega \sum_{j=1}^n \eta^2 \langle D^2 L(\nabla u) \nabla u_{x_j}, \nabla \gamma^j \rangle
& = - \int_\Omega \sum_{j=1}^n \gamma^j \langle D^2 L(\nabla u) \nabla u_{x_j}, \nabla \eta^2 \rangle \\
& \phantom{=} - \int_\Omega \sum_{j=1}^n \eta^2 \gamma^j f_{x_j}. 
\end{align*}
Let us denote the integrands above by \barroman{I}, \barroman{II}, \barroman{III}. 

\medskip

The estimate for the left-hand side is crucial. A straightforward calculation based on Lemma \ref{lem:L-properties} leads to 
\begin{align*}
\left( D^2 L(\nabla u) \nabla u_{x_j} \right)^i & = 
l(\nabla u)^{p-2} u_{x_i x_j} + (p-2) l(\nabla u)^{p-4} \langle \nabla u_{x_j}, \nabla u \rangle u_{x_i}, \\
\frac{\partial}{\partial x_i} \gamma^j & = 
l(\nabla u)^{2-q} u_{x_i x_j} + (2-q) l(\nabla u)^{-q} \langle \nabla u_{x_i}, \nabla u \rangle u_{x_j},
\end{align*}
which gives us 
\begin{align*}
\barroman{I}
& = \eta^2 l(\nabla u)^{p-q} \left ( \left| D^2 u \right|^2 
+ (p-q) \left| D^2 u \cdot \frac{\nabla u}{l(\nabla u)} \right|^2 \right. \\
& \phantom{=} \left. - (p-2)(q-2) \left| \left\langle D^2 u \cdot \frac{\nabla u}{l(\nabla u)}, \frac{\nabla u}{l(\nabla u)} \right\rangle \right|^2 \right ) \\
& \ge \min \left( 1, (p-1)(3-q) \right) \cdot \eta^2 l(\nabla u)^{p-q} \left| D^2 u \right|^2 
\end{align*}
In the last line we used the inequality $|\nabla u| \le l(\nabla u)$, the Cauchy-Schwarz inequality 
\[ 
\left| \left\langle D^2 u \cdot v, v \right\rangle \right| 
\le 
\left| D^2 u \cdot v \right| 
\le 
|D^2 u| 
\qquad \text{for any vector } |v| \le 1 
\]
and our choice of $q$: 
\[
1 + (p-q) - (p-2)(q-2) = (p-1)(3-q) > 0. 
\]

\medskip

The right-hand side is estimated in the standard way using Young's inequality: 
\begin{align*}
\barroman{II}
& \le \sum_{j=1}^n |\gamma^j| \cdot |D^2 L(\nabla u) \nabla u_{x_j}| \cdot |\nabla \eta^2| \\
& \le C(\Omega,\Omega',n,p,q) \cdot \eta l(\nabla u)^{p-q+1} |D^2 u| \\
& \le \delta C \cdot \eta^2 l(\nabla u)^{p-q} |D^2 u|^2 + \frac{1}{\delta} C \cdot l(\nabla u)^{p-q+2}. 
\end{align*}
For small enough $\delta > 0$, the first term can be absorbed by the left-hand side and the second is bounded using H\"{o}lder's inequality 
\begin{equation}
\label{eq:V-L2}
\int_\Omega l(\nabla u)^{p-q+2} 
\le |\Omega|^{\frac{q-2}{p}} \left ( \int_\Omega l(\nabla u)^p \right )^{\frac{p-q+2}{p}}
\le C(\Omega,p,q,\| u_0 \|_{W^{1,p}(\Omega)}), 
\end{equation}
where the second inequality above was shown in the proof of Lemma \ref{lem:unif-bounds}a. The last term is similar: 
\begin{equation}
\label{eq:f-estimate}
\barroman{III} 
\le \left( \int_\Omega l(\nabla u)^{p(3-q)} \right)^{1/p} \| \nabla f \|_{L^{p'}(\Omega)}
\le C(\Omega,p,q,\| u_0 \|_{W^{1,p}(\Omega)},\| f_0 \|_{W^{1,p'}(\Omega)}). 
\end{equation}
Note that one could apply H\"{o}lder's inequality with exponents $(\frac{p}{3-q},\frac{p}{p+q-3})$ instead of $(p,p')$, thus using weaker estimates on $f_0$. 

\medskip

Recalling $\eta \equiv 1$ on $\Omega'$, we can summarize these estimates with 
\begin{equation}
\label{eq:V_W12}
\int_{\Omega'} l(\nabla u)^{p-q} |D^2 u| \le C(\Omega,\Omega',n,p,q,\| u_0 \|_{W^{1,p}(\Omega)},\| f_0 \|_{W^{1,p'}(\Omega)}), 
\end{equation}
where the constant may depend on everything except $\varepsilon$. 

The function $V := \alpha^s(\nabla u) = l(\nabla u)^{s-1} \nabla u$ is smooth as a composition of smooth functions. Since $|V| \le l(\nabla u)^{s}$, the $L^2(\Omega)$-norm of $V$ has been estimated in \eqref{eq:V-L2}. Similarly, $|\nabla V| \le C(n,s) l(\nabla u)^{s-1} |D^2 u|^2$, hence \eqref{eq:V_W12} gives a bound on $\| \nabla V \|_{L^2(\Omega')}$ and finishes the proof. 
\end{proof}

\section{Fractional differentiability}
\label{sec:fract-diff}

\begin{lem}[fractional differentiability lemma]
\label{lem:fract-diff}
Assume that $\beta \colon \R^k \to \R^k$ is H{\"o}lder continuous with exponent $\theta \in (0,1)$ and constant $M > 0$, i.e. 
\[ |\beta(w) - \beta(v)| \le M |w-v|^\theta \quad \text{for } w,v \in \R^k. \]
If $V \in W^{1,2}(\Omega,\R^k)$ and $\Omega' \Subset \Omega$, then $\beta(V) \in \cN^{\theta,2/\theta}(\Omega',\R^k)$ with 
\[ [ \beta(V) ]_{\cN^{\theta,2/\theta}(\Omega')} \le C(n) M [V]^\theta_{W^{1,2}(\Omega)}. \]
\end{lem}

\begin{proof}
Choose $\delta = \dist(\Omega',\partial \Omega)$ and fix some vector $v \in \R^n$ of length $|v| \le \delta$. For any $x \in \Omega'$, 
\begin{align*}
|\beta(V(x+v)) - \beta(V(x))|^{2/\theta} 
& \le (M|V(x+v)-V(x)|^\theta)^{2/\theta} \\
& = M^{2/\theta} |V(x+v)-V(x)|^2. 
\end{align*}
Integrating the above over $\Omega'$ yields 
\begin{align*}
\int_{\Omega'} |\beta(V(x+v)) - \beta(V(x))|^{2/\theta} 
& \le M^{2/\theta} \int_{\Omega'} |V(x+v)-V(x)|^2 \\
& \le C(n) M^{2/\theta} [V]_{W^{1,2}(\Omega)}^2 |v|^2.
\end{align*}
\end{proof}

\begin{proof}[Proof of Theorem \ref{th:p-harmonic-regularity}]
Choose $\theta \in [\frac{2}{p},\frac{2}{p-1})$ and $s = 1/\theta$. Then by Theorem \ref{th:nonlinear-reg} $V = |\nabla u_0|^{s-1} \nabla u_0 \in W_\loc^{1,2}(\Omega)$. Introduce the function 
\[ 
\beta \colon \R^n \to \R^n, \qquad 
\beta(w) = |w|^{\theta-1} w, 
\] 
so that $u_0 = \beta(V)$. To see that it is H\"{o}lder continuous, consider its inverse -- the elementary inequality (see e.g. \cite[Ch.~10]{Lin})
\[
\left\langle |w|^{s-1} w - |v|^{s-1} v, w-v \right\rangle 
\ge \frac 12 \left( |w|^{s-1}+|v|^{s-1} \right) |w-v|^2
\]
implies that $||w|^{s-1} w - |v|^{s-1} v| \ge 2^{-s} |w-v|^s$ and in consequence $\beta$ is $\theta$-H\"{o}lder continuous with constant $2$. By Lemma \ref{lem:fract-diff}, this implies $\nabla u_0 \in \cN_\loc^{\theta,2/\theta}(\Omega)$ together with the desired estimates. 
\end{proof}

\section{Sharpness of the estimates}
\label{sec:sharpness}

Let $p \ge 3$ and $u(x_1,\ldots,x_n) = \frac{1}{p'} |x_1|^{p'}$. It is easily seen that $u$ solves the inhomogenous $p$-Laplace equation $\dv(|\nabla u|^{p-2} \nabla u) = 1$ in $\R^n$. For fixed $q \in [1,\infty]$ we can find the largest $\theta > 0$ for which $u \in \cN^{\theta,q}(\B(0,1))$, arriving at 
\[
\nabla u \in 
\begin{cases}
C^{0,\frac{1}{p-1}} & \text{ for } q = \infty, \\
\cN^{\theta,q}, \ \theta = \frac{1}{p-1}+\frac{1}{q} & \text{ for } \frac{p-1}{p-2} < q < \infty, \\
W^{1,q} & \text{ for } 1 \le q < \frac{p-1}{p-2}. 
\end{cases}
\]
As a special case, we note that $\nabla u \in \cN^{\frac{2}{p-1},p-1}$ for $p > 3$ (but not for $p=3$). Moreover, this is optimal in the sense that $\nabla u \notin \cN^{\frac{2}{p-1},q}$ for $q > p-1$ and $\nabla u \notin \cN^{\theta,p-1}$ for $\theta > \frac{2}{p-1}$. 

It is natural to ask whether the claim of Theorem \ref{th:p-harmonic-regularity} can be strengthened to cover the endpoint case $\theta = \frac{2}{p-1}$ for $p > 3$. However, in view of this example one cannot hope for more regularity. 

\bibliography{p-harmonic}

\begin{thebibliography}{10}

\bibitem{Ada}
{\sc Adams, R.~A.}
\newblock {\em Sobolev spaces}.
\newblock Academic Press [A subsidiary of Harcourt Brace Jovanovich,
  Publishers], New York-London, 1975.
\newblock Pure and Applied Mathematics, Vol. 65.

\bibitem{BojIwa}
{\sc Bojarski, B., and Iwaniec, T.}
\newblock {$p$}-harmonic equation and quasiregular mappings.
\newblock In {\em Partial differential equations ({W}arsaw, 1984)}, vol.~19 of
  {\em Banach Center Publ.} PWN, Warsaw, 1987, pp.~25--38.

\bibitem{Cel}
{\sc Cellina, A.}
\newblock The regularity of solutions to some variational problems, including
  the $p$-{L}aplace equation for $2 \le p < 3$.
\newblock {\em ESAIM: COCV\/} (2016).

\bibitem{DiB}
{\sc DiBenedetto, E.}
\newblock {$C^{1+\alpha }$}\ local regularity of weak solutions of degenerate
  elliptic equations.
\newblock {\em Nonlinear Anal. 7}, 8 (1983), 827--850.

\bibitem{Eva82}
{\sc Evans, L.~C.}
\newblock A new proof of local {$C^{1,\alpha }$}\ regularity for solutions of
  certain degenerate elliptic p.d.e.
\newblock {\em J. Differential Equations 45}, 3 (1982), 356--373.

\bibitem{LadUra}
{\sc Ladyzhenskaya, O.~A., and Ural'tseva, N.~N.}
\newblock {\em Linear and quasilinear elliptic equations}.
\newblock Translated from the Russian by Scripta Technica, Inc. Translation
  editor: Leon Ehrenpreis. Academic Press, New York-London, 1968.

\bibitem{Lew}
{\sc Lewis, J.~L.}
\newblock Regularity of the derivatives of solutions to certain degenerate
  elliptic equations.
\newblock {\em Indiana Univ. Math. J. 32}, 6 (1983), 849--858.

\bibitem{Lin}
{\sc Lindqvist, P.}
\newblock {\em Notes on the {$p$}-{L}aplace equation}, vol.~102 of {\em Report.
  University of Jyv\"askyl\"a Department of Mathematics and Statistics}.
\newblock University of Jyv\"askyl\"a, Jyv\"askyl\"a, 2006.

\bibitem{Min}
{\sc Mingione, G.}
\newblock The singular set of solutions to non-differentiable elliptic systems.
\newblock {\em Arch. Ration. Mech. Anal. 166}, 4 (2003), 287--301.

\bibitem{Nik}
{\sc Nikol'ski\u\i, S.~M.}
\newblock {\em Approximation of functions of several variables and imbedding
  theorems}.
\newblock Springer-Verlag, New York-Heidelberg., 1975.
\newblock Translated from the Russian by John M. Danskin, Jr., Die Grundlehren
  der Mathematischen Wissenschaften, Band 205.

\bibitem{Sci}
{\sc Sciunzi, B.}
\newblock Regularity and comparison principles for {$p$}-{L}aplace equations
  with vanishing source term.
\newblock {\em Commun. Contemp. Math. 16}, 6 (2014), 1450013, 20.

\bibitem{Tol83}
{\sc Tolksdorf, P.}
\newblock Everywhere-regularity for some quasilinear systems with a lack of
  ellipticity.
\newblock {\em Ann. Mat. Pura Appl. (4) 134\/} (1983), 241--266.

\bibitem{Tol84}
{\sc Tolksdorf, P.}
\newblock Regularity for a more general class of quasilinear elliptic
  equations.
\newblock {\em J. Differential Equations 51}, 1 (1984), 126--150.

\bibitem{Uhl}
{\sc Uhlenbeck, K.}
\newblock Regularity for a class of non-linear elliptic systems.
\newblock {\em Acta Math. 138}, 3-4 (1977), 219--240.

\bibitem{Ura}
{\sc Ural'tseva, N.~N.}
\newblock Degenerate quasilinear elliptic systems.
\newblock {\em Zap. Nau\v cn. Sem. Leningrad. Otdel. Mat. Inst. Steklov. (LOMI)
  7\/} (1968), 184--222.

\end{thebibliography}
\bibliographystyle{acm}

\end{document}